\documentclass[11pt]{article}

\usepackage[T1]{fontenc}
\usepackage{lmodern}
\usepackage{microtype}
\usepackage{amsmath,amssymb,amsthm,mathtools}
\usepackage{xcolor}
\usepackage[margin=1.15in]{geometry}
\usepackage{enumitem}
\usepackage{array,booktabs}
\usepackage{authblk}
\usepackage{tikz}
\usetikzlibrary{arrows.meta,positioning}
\usepackage[colorlinks=true,linkcolor=blue!55!black,citecolor=blue!55!black,
  urlcolor=blue!55!black]{hyperref}

\newtheorem{theorem}{Theorem}[section]
\newtheorem{lemma}[theorem]{Lemma}
\newtheorem{proposition}[theorem]{Proposition}
\newtheorem{corollary}[theorem]{Corollary}
\newtheorem{remark}[theorem]{Remark}
\newtheorem{fact}[theorem]{Fact}
\theoremstyle{definition}
\newtheorem{definition}[theorem]{Definition}
\theoremstyle{plain}
\newtheorem*{mainrestated}{Theorem~\ref{thm:main} (restated)}
\numberwithin{table}{section}
\numberwithin{figure}{section}

\hypersetup{
  pdftitle={Weak Log-Majorization for Negative Lim-Palfia Power Means},
  pdfauthor={Marco Tomamichel},
  pdfsubject={Matrix power means and weak log-majorization},
  pdfkeywords={matrix power mean, weak log-majorization, Schatten quasi-norm},
}

\newcommand{\M}{\mathbb{M}}
\newcommand{\Tr}{\operatorname{Tr}}
\newcommand{\diag}{\operatorname{diag}}

\newcommand{\Sshort}{\mathcal{S}}
\newcommand{\uinorm}[1]{%
  \lvert\!\lvert\!\lvert #1\rvert\!\rvert\!\rvert}
\newcommand{\numberedqedstack}[1]{%
  \vtop{\ialign{\hfil##\hfil\cr
    \textup{(#1)}\cr
    \qedsymbol\cr}}}

\title{Weak Log-Majorization for Negative Lim--P\'alfia Power Means}
\author{Marco Tomamichel\thanks{Email: \texttt{marco.tomamichel@nus.edu.sg}}}
\affil{Department of Electrical and Computer Engineering \& Centre for Quantum
Technologies, National University of Singapore, Singapore}
\date{\today}

\begin{document}
\maketitle

\begin{abstract}
For $0\le\alpha\le1$ and $-1\le r<0$, let $\#_{r,\alpha}$ denote the
weighted Kubo--Ando power mean of order $r$ and weight $\alpha$.  Ando proved
that, for every unitarily invariant norm,
\[
  \uinorm{A\mathbin{\#_{r,\alpha}}B}
  \le
  \uinorm{A}\mathbin{\#_{r,\alpha}}\uinorm{B}.
\]
We show that this inequality remains valid also for the
Schatten $q$-quasi-norms with $0<q<1$.  More generally, our main result
resolves a problem recently posed by Hiai and Lim in the stronger form of weak
log-majorization for multivariable Lim--P\'alfia power means of negative
order. It yields analogous Schatten quasi-norm inequalities for every
finite family of positive definite matrices.
\end{abstract}



\paragraph{Use of artificial intelligence.}
OpenAI Codex with ChatGPT 5.6 Sol was instrumental for literature searches, exploration of proof
strategies, preparation and revision of the manuscript, and development of
the Lean formalization. Claude Fable 5 acted as adversarial reviewer and gave useful feeback. All definitions, statements, proofs and references have been verified by the author, who takes responsibility for
any remaining errors.

\paragraph{Formal verification.}
The proofs of the main result and corollary are formalized in
Lean~4.32.2 \cite{Lean4} using Mathlib~4.32.2 \cite{Mathlib} and checked
by the Lean kernel. The source code is available on GitHub~\cite{TomamichelFormalization}. 
Appendix~\ref{app:lean-correspondence} gives the statement-by-statement
correspondence with the manuscript.

\paragraph*{Acknowledgements.}
MT is supported by the NRF Investigatorship award (NRF-NRFI10-2024-0006).

\newpage

\section{Introduction}

For a positive probability vector $\mathbf w=(w_i)_{i=1}^m$ and a positive
vector $\mathbf a=(a_j)_{j=1}^m$, the scalar power mean of order $r\ne0$
is
\begin{equation}\label{eq:scalar-power-mean}
  p_r(\mathbf w;\mathbf a)
  :=\left(\sum_{j=1}^m w_ja_j^r\right)^{1/r}.
\end{equation}
It interpolates between the weighted harmonic and arithmetic means as
$r$ ranges from $-1$ to $1$, and its limit at $r=0$ is the weighted
geometric mean.  Extending this family to noncommuting positive definite
matrices is not a matter of simply substituting matrices in
\eqref{eq:scalar-power-mean}.  For example, the
expression $(\sum_jw_jA_j^r)^{1/r}$ tends as $r\to0$ to the
Log--Euclidean mean $\exp(\sum_j w_j \log A_j)$, which generally differs
from the Riemannian barycenter, or Karcher mean.

Lim and P\'alfia introduced their matrix power means to obtain
a noncommutative analogue that retains the geometric structure of the
positive definite cone \cite{LimPalfia}.  The construction starts from
the scalar fixed-point identity
$x=\sum_{j=1}^m w_jx^{1-t}a_j^t$ and replaces the product
$x^{1-t}a_j^t$ by the weighted matrix geometric mean, 
\begin{equation}
	A\#_tB:=A^{1/2}(A^{-1/2}BA^{-1/2})^tA^{1/2} \,.
\end{equation}

\begin{definition}[Lim--P\'alfia power means]\label{def:power-means}
Let $\mathbf w=(w_i)_{i=1}^m$ be a positive probability vector and let
$\mathbf A=(A_j)_{j=1}^m$ be an $m$-tuple of positive definite matrices.
For $0<t\le1$, the equation
\begin{equation}\label{eq:positive-fixed-point}
  X=\sum_{j=1}^m w_j\,(X\#_tA_j).
\end{equation}
has a unique positive definite solution
\cite[Theorem~3.1 and Definition~3.2]{LimPalfia}.  The positive-order
Lim--P\'alfia power mean $P_t(\mathbf w;\mathbf A)$ is defined to be this
solution.  Write $\mathbf A^{-1}:=(A_j^{-1})_{j=1}^m$.  For
$-1\le r<0$, the power mean is defined by duality as
\begin{equation}\label{eq:negative-duality}
  P_r(\mathbf w;\mathbf A)
  :=P_{-r}(\mathbf w;\mathbf A^{-1})^{-1}.
\end{equation}
\end{definition}

The endpoints $P_{-1}$ and $P_1$ are respectively the weighted harmonic
and arithmetic means \cite[Remark~3.3]{LimPalfia}.  The original purpose
of the construction was to furnish a family of matrix means converging
to the Karcher mean as $r\to0$ \cite[Theorem~4.3]{LimPalfia}.

The question considered here asks how this genuinely noncommutative mean
compares with the ordinary scalar power mean at the level of spectra.
For a Hermitian matrix $A$, write
$\lambda_1^{\downarrow}(A)\ge\cdots\ge\lambda_n^{\downarrow}(A)$ for its
eigenvalues in decreasing order and set
$\lambda_i^\downarrow(\mathbf A):=
(\lambda_i^\downarrow(A_j))_{j=1}^m$.  
The question belongs to a broader program of Hiai and Lim~\cite{HiaiLim}. Their
weak-majorization result for positive power means states that, for
$0<r\le1$,\footnote{This is Proposition~7.5 of the arXiv version of their paper~\cite{HiaiLimArxiv}; that section is
not in the published version.}
\begin{equation}\label{eq:positive-weak-majorization}
\left( \lambda_i^\downarrow  \bigl(P_r(\mathbf w;\mathbf A)\bigr) \right)_{i=1}^n
 \prec_w
 \left(
   p_r\bigl(\mathbf w;\lambda_i^\downarrow(\mathbf A)\bigr)
 \right)_{i=1}^n.
\end{equation}
Here weak majorization, $x\prec_w y$, means that the sum of the $k$ largest entries of $x$
does not exceed the corresponding sum for $y$, for every $k$. Hiai and Lim observed that
duality gives a complementary inverse inequality for negative orders,
but not the desired weak majorization, and posed the latter for
$-1<r<0$ as an open problem.  To state our result,
recall that for positive vectors $x,y\in(0,\infty)^n$, one writes
$x\prec_{w\log}y$ if
$\prod_{i=1}^k x_i^{\downarrow}\le\prod_{i=1}^k y_i^{\downarrow}$ for
$1\le k\le n$.
Weak log-majorization implies weak majorization, i.e., $x\prec_{w\log}y \Longrightarrow x\prec_wy$.
Our main result shows that weak log-majorization holds for power means of
negative order.

\begin{theorem}\label{thm:main}
Let $\mathbf w=(w_i)_{i=1}^m$ be a positive probability vector and let
$\mathbf A=(A_j)_{j=1}^m$ be an $m$-tuple of positive definite matrices
in $\M_n$.  For every $-1\le r<0$,
\begin{equation}\label{eq:main}
\left( \lambda_i^\downarrow  \bigl(P_r(\mathbf w;\mathbf A)\bigr) \right)_{i=1}^n
 \prec_{w\log}
 \left(
   p_r\bigl(\mathbf w;\lambda_i^{\downarrow}(\mathbf A)\bigr)
 \right)_{i=1}^n.
\end{equation}
\end{theorem}

The weak log-majorization \eqref{eq:main} is sharp already in the
commutative setting.  When the $A_j$ are simultaneously diagonalizable
and their decreasing eigenvalue lists are aligned in a common eigenbasis,
the fixed-point equation is scalar in each coordinate and equality holds
in \eqref{eq:main}.

Consequently, the weak majorization in the Hiai--Lim preprint
\cite[Problem~7.7]{HiaiLimArxiv} holds in this stronger form. We note that for $r > 0$, this
cannot be strengthened to weak log-majorization, even for two
commuting $2\times2$ matrices. A counterexample is obtained by taking
$B_1=\diag(a,b)$ and $B_2=\diag(b,a)$ for $a>b>0$, and
$\mathbf w=(\tfrac12,\tfrac12)$.  In this case
$P_r(\mathbf w;(B_1,B_2))=p_r(\mathbf w;(a,b))I_2$, whereas the scalar
comparator is $(a,b)$.  The full-product inequality would therefore require
$p_r(\mathbf w;(a,b))^2\le ab$, contrary to
$p_r(\mathbf w;(a,b))>p_0(\mathbf w;(a,b))=\sqrt{ab}$ for $r>0$.
There is a related multivariable result of Jeong and Kim
\cite[Proposition~3.3]{JeongKim}: for $-1\le r<0$, they showed that the
Lim--P\'alfia power mean is weakly log-majorized by the Log--Euclidean
mean. 

\begin{remark}[Extension to probability measures]
\label{rem:probability-measures}
The same argument should extend from finite support to the bounded-support
probability measures considered by Hiai and Lim, using their approximation
and continuity framework; see
\cite[Theorem~2.9, Section~5.2, and Remark~5.2]{HiaiLim} for the relevant
methods.  We do not pursue the measure-theoretic formulation here.
\end{remark}

One consequence of weak log-majorization is the following inequality for quasi-norms.

\begin{corollary}[Schatten quasi-norm inequality]
\label{cor:schatten-quasinorm}
Let $\mathbf w=(w_i)_{i=1}^m$ be a positive probability vector and let
$\mathbf A=(A_j)_{j=1}^m$ be an $m$-tuple of positive definite matrices
in $\M_n$.  For every $-1\le r<0$ and $q>0$,
\begin{equation}\label{eq:schatten-quasinorm}
 \bigl\|P_r(\mathbf w;\mathbf A)\bigr\|_q
 \le
 p_r\bigl(\mathbf w;(\|A_j\|_q)_{j=1}^m\bigr),
\end{equation}
where $\|\cdot\|_q$ denotes the Schatten $q$-norm for $q\ge1$ and the
Schatten $q$-quasi-norm for $0<q<1$.
\end{corollary}

For $q\ge1$, Corollary~\ref{cor:schatten-quasinorm} is contained in~\cite[Proposition~7.2]{HiaiLim}.  Its new content
is the quasi-norm range $0<q<1$.

For two inputs, $\mathbf w=(1-\alpha,\alpha)$ and $\mathbf A=(A,B)$, the
Lim--P\'alfia mean is the weighted Kubo--Ando power mean
\begin{equation}\label{eq:binary-power-mean}
 P_r\bigl((1-\alpha,\alpha);(A,B)\bigr)
 =A\mathbin{\#_{r,\alpha}}B :=A^{1/2}\left[(1-\alpha)I+
   \alpha\bigl(A^{-1/2}BA^{-1/2}\bigr)^r\right]^{1/r}A^{1/2}.
\end{equation}
Congruence normalization of the fixed-point equation gives
\eqref{eq:binary-power-mean} for positive order.  For $r<0$, apply that
identity to $A^{-1}$ and $B^{-1}$ and use binary duality,
$(A\mathbin{\#_{r,\alpha}}B)^{-1}
=A^{-1}\mathbin{\#_{-r,\alpha}}B^{-1}$.
Ando's inequality for operator means
\cite[Eq.~(3.13)]{AndoMajorizations} gives the unitarily invariant norm case;
Corollary~\ref{cor:schatten-quasinorm} extends it to Schatten quasi-norms.

For related two-variable spectral-product questions and operator power
means, see Bourin and Hiai
\cite[Remark~3.4 and Example~3.11]{BourinHiai}.

The proof has two problem-specific ingredients.  First, a multiplicative
Lidskii inequality yields a determinant estimate pairing the
correspondingly ordered eigenvalues of all inputs.  We prove this estimate
directly for positive supersolutions of the power-mean equation.  Second, shorting to
the bottom spectral subspace of the power mean converts this estimate
into inequalities for every lower partial eigenvalue product. 


\section{Preliminaries}\label{sec:standard-tools}

The proof uses the Lim--P\'alfia fixed-point result recorded in
Definition~\ref{def:power-means}, Ando's inequality for positive maps, the
lower multiplicative Lidskii inequality, and Cauchy's interlacing theorem.
We state the latter three results here with the precise forms and references
needed later.

\subsection{Ando's positive-map inequality}

Let $V:\mathbb C^k\to\mathbb C^n$ be an isometry, so $V^*V=I_k$.
The compression map $\Phi(Z)=V^*ZV$ is unital and completely positive.
The following result compares the geometric mean before and after such
a compression.

\begin{fact}[Ando's compression inequality]
\label{fact:ando-compression}
For positive definite $A,B\in\M_n$, an isometry
$V:\mathbb C^k\to\mathbb C^n$, and $0\le t\le1$,
\begin{equation}\label{eq:ando-compression}
  V^*(A\#_tB)V
  \le (V^*AV)\#_t(V^*BV).
\end{equation}
This is Ando's Jensen inequality for normalized positive linear maps
\cite[Theorem~4]{AndoPositiveMaps}, applied to the compression
$\Phi(Z)=V^*ZV$ and the operator-concave function $x\mapsto x^t$.
\end{fact}

The direction in \eqref{eq:ando-compression} reflects operator
concavity.  Later we apply it to inverse matrices and then invert the
result; this produces the corresponding inequality for shorted
geometric means.

\subsection{Multiplicative Lidskii inequality}

\begin{fact}[Lower multiplicative Lidskii inequality]
\label{fact:multiplicative-lidskii}
If $A,B\in\M_n$ are positive definite, then
\begin{equation}\label{eq:multiplicative-lidskii}
 \bigl(\log\lambda_i^{\downarrow}(A)
       -\log\lambda_i^{\downarrow}(B)\bigr)_{i=1}^n
 \prec
 \bigl(\log\lambda_i^{\downarrow}
   (B^{-1/2}AB^{-1/2})\bigr)_{i=1}^n,
\end{equation}
where the vector on the left is rearranged decreasingly in the definition
of majorization.  The partial-sum inequalities in
\eqref{eq:multiplicative-lidskii} are the positive definite, logarithmic
form of Li and Mathias
\cite[Theorem~2.3, upper inequality in Eq.~(2.4)]{LiMathias}.  For the full
sum, both sides are equal to $\log\det A-\log\det B$.  Thus
\eqref{eq:multiplicative-lidskii} is an ordinary majorization relation.
\end{fact}

The inequality pairs the decreasing eigenvalues of $A$ and $B$ before
comparing them with the spectrum of the relative matrix
$B^{-1/2}AB^{-1/2}$.  Karamata's inequality for the exponential function
turns this logarithmic majorization directly into the trace estimate used
below.

\subsection{Cauchy's interlacing theorem}

For an isometry $V:\mathbb C^k\to\mathbb C^n$, the matrix $V^*HV$ is the
compression of $H$ to the range of $V$, after identifying that range with
$\mathbb C^k$.  Cauchy's interlacing theorem gives the following bounds.

\begin{fact}[Cauchy's interlacing theorem]
\label{fact:cauchy-interlacing}
Let $H\in\M_n$ be Hermitian and let
$V:\mathbb C^k\to\mathbb C^n$ be an isometry.  Then, for $1\le i\le k$,
\begin{equation}\label{eq:cauchy-interlacing}
 \lambda_i^\downarrow(H)
 \ge \lambda_i^\downarrow(V^*HV)
 \ge \lambda_{n-k+i}^\downarrow(H).
\end{equation}
This is Bhatia \cite[\S26, Theorem~26.1]{BhatiaPerturbation}.
\end{fact}


\section{A determinant inequality for positive power means}
\label{sec:determinant}

Recall the fixed-point equation \eqref{eq:positive-fixed-point} for positive power means.  We call a positive definite matrix $X$
a \emph{supersolution} of this equation if
\begin{equation}\label{eq:power-supersolution}
X\ge\sum_{j=1}^m w_j(X\#_tB_j)
\end{equation}
in the Loewner order. The following rephrasing (normalisation) of this inequality will be useful.

\begin{lemma}[Supersolution inequality]
\label{lem:normalized-supersolution}
Let $\mathbf w=(w_i)_{i=1}^m$ be a positive probability vector, let
$\mathbf B=(B_j)_{j=1}^m$ be an $m$-tuple of positive definite matrices,
and let $t>0$.  Suppose that
$X\ge\sum_{j=1}^m w_j(X\#_tB_j)$, and put
$D_j:=X^{-1/2}B_jX^{-1/2}$.  Then
\begin{equation} \label{eq:normalized-supersolution} 
  \sum_{j=1}^m w_jD_j^t\le I, \qquad \textnormal{and, consequently}, \qquad 
  n\ge\sum_{j=1}^m w_j\Tr D_j^t.
\end{equation}
\end{lemma}

\begin{proof}
By the definition of the weighted geometric mean and the transformer identity,
\begin{equation}
	X^{-1/2}(X\#_tB_j)X^{-1/2}=D_j^t \,.
\end{equation}
Congruence by $X^{-1/2}$ preserves the Loewner order and sends $X$ to
$I$.  Applying it to the assumed supersolution inequality and taking the trace therefore
gives \eqref{eq:normalized-supersolution}.
\end{proof}

For an exact fixed point, the calculation gives equality in
\eqref{eq:normalized-supersolution}; in particular, it recovers the
discrete specialization of Hiai--Lim \cite[Equation (5.4)]{HiaiLim}.  Only the
one-sided form is needed here.  The next lemma converts the resulting
trace term into an eigenvalue-paired scalar bound.

\begin{lemma}[Trace consequence of multiplicative Lidskii]
\label{lem:lidskii-trace}
Let $A,B\in\M_n$ be positive definite, let $t>0$, and write
$a_i:=\lambda_i^\downarrow(A)$ and $b_i:=\lambda_i^\downarrow(B)$.  Then
\begin{equation}\label{eq:lidskii-trace}
  \Tr\bigl(B^{-1/2}AB^{-1/2}\bigr)^t
  \ge\sum_{i=1}^n\left(\frac{a_i}{b_i}\right)^t.
\end{equation}
\end{lemma}

\begin{proof}
Put $D:=B^{-1/2}AB^{-1/2}$ and define $u=(u_i)_{i=1}^n$ and
$v=(v_i)_{i=1}^n$ by $u_i:=\log a_i-\log b_i$ and
$v_i:=\log\lambda_i^\downarrow(D)$.
Fact~\ref{fact:multiplicative-lidskii} says that
$u^\downarrow\prec v$, where $u^\downarrow$ is the decreasing
rearrangement of $u$.  Recall Karamata's inequality: if
$x,y\in\mathbb R^n$ satisfy $x\prec y$ and $f$ is convex on an interval
containing their entries, then
$\sum_{i=1}^nf(x_i)\le\sum_{i=1}^nf(y_i)$; see Bhatia
\cite[Theorem~II.3.1]{Bhatia}.  Since $x\mapsto e^{tx}$ is convex for
$t>0$, Karamata applied to $u^\downarrow\prec v$ gives
\begin{align}
  \sum_{i=1}^n\left(\frac{a_i}{b_i}\right)^t
  &=\sum_{i=1}^n e^{tu_i}
   =\sum_{i=1}^n e^{tu_i^\downarrow}\\
  &\le\sum_{i=1}^n e^{tv_i}
   =\sum_{i=1}^n\bigl[\lambda_i^\downarrow(D)\bigr]^t
   =\Tr D^t.\label{eq:lidskii-trace-calculation}
\end{align}
The middle equality in the first line holds because rearranging the
entries of $u$ does not change their exponential sum.
\end{proof}

Combining this trace estimate with the normalized supersolution
inequality gives the determinant bound that drives the rest of the proof.

\begin{proposition}[Eigenvalue-paired determinant bound for supersolutions]
\label{prop:determinant}
Let $\mathbf w=(w_i)_{i=1}^m$ be a positive probability vector, let
$\mathbf B=(B_j)_{j=1}^m$ be an $m$-tuple of positive definite matrices
in $\M_n$, let $t>0$, and suppose that $X>0$ is a supersolution in the
sense of \eqref{eq:power-supersolution}.
If $b_{j,1}\ge\cdots\ge b_{j,n}$ are the eigenvalues of $B_j$ and
$\mathbf b_i:=(b_{j,i})_{j=1}^m$, then
\begin{equation}\label{eq:determinant-bound}
  \det X
  \ge
  \prod_{i=1}^n p_t(\mathbf w;\mathbf b_i).
\end{equation}
\end{proposition}

\begin{proof}
Write $x_1\ge\cdots\ge x_n$ for the eigenvalues of $X$, and set
$D_j:=X^{-1/2}B_jX^{-1/2}$.

We first use the supersolution condition.  Lemma~\ref{lem:normalized-supersolution},
applied to $X$ and the matrices $B_j$, gives
$n\ge\sum_jw_j\Tr D_j^t$.  Next fix $j$.  The decreasingly ordered
eigenvalues of $B_j$ and $X$ are respectively $b_{j,i}$ and $x_i$, so
Lemma~\ref{lem:lidskii-trace}, with $A=B_j$ and $B=X$, yields
$\Tr D_j^t\ge\sum_{i=1}^n(b_{j,i}/x_i)^t$.  Multiplying these
inequalities by $w_j$ and summing over $j$ therefore gives
\begin{align}
 n
 &\ge\sum_{j=1}^m w_j\Tr D_j^t
 \ge\sum_{j=1}^m w_j\sum_{i=1}^n
     \left(\frac{b_{j,i}}{x_i}\right)^t
     \label{eq:aggregated-trace-bound}\\
 &=\sum_{i=1}^n\frac{\sum_jw_jb_{j,i}^t}{x_i^t}
 =\sum_{i=1}^n
   \left(\frac{p_t(\mathbf w;\mathbf b_i)}{x_i}\right)^t.
   \label{eq:aggregated-trace-bound-final}
\end{align}
Here we interchanged the two finite sums and then used
$p_t(\mathbf w;\mathbf b_i)^t=\sum_jw_jb_{j,i}^t$.

The chain \eqref{eq:aggregated-trace-bound}--\eqref{eq:aggregated-trace-bound-final},
after division by $n$, says that
the arithmetic mean of the positive numbers
$\bigl(p_t(\mathbf w;\mathbf b_i)/x_i\bigr)^t$ is at most $1$.
The scalar arithmetic--geometric mean inequality therefore yields
\begin{equation}
1\ge \left( \prod_{i=1}^n
(p_t(\mathbf w;\mathbf b_i)/x_i)^t \right)^{1/n} \,.
\end{equation}
Since $t>0$, this is equivalent to
$\prod_i p_t(\mathbf w;\mathbf b_i)\le\prod_i x_i=\det X$, which is
\eqref{eq:determinant-bound}.
\end{proof}

\section{Lower spectral products}
\label{sec:lower-products}

We now lift the determinant inequality to every lower partial product.
The localization device is the shorted matrix.  Let
$V:\mathbb C^k\to\mathbb C^n$ be an isometry and, for $A>0$, write
$\Sshort_V(A):=(V^*A^{-1}V)^{-1}$.
This is the matrix of the shorted operator on the range of $V$.  If
$[V\ W]$ is unitary and
\begin{equation}
[V\ W]^*A[V\ W]
=\left[\begin{smallmatrix}A_{11}&A_{12}\\A_{21}&A_{22}\end{smallmatrix}\right],
\end{equation}
then the block inversion formula gives
\begin{equation}\label{eq:schur-complement}
  \Sshort_V(A)=A_{11}-A_{12}A_{22}^{-1}A_{21}.
\end{equation}
Shorted operators were developed systematically by Anderson and Trapp
\cite{AndersonTrapp}.  We need three elementary order and spectral
properties of this operation.  The first says that shorting a weighted sum
dominates the corresponding weighted sum of the shorted matrices.

\begin{lemma}[Weighted superadditivity]
\label{lem:short-superadditive}
If $\mathbf A=(A_j)_{j=1}^m$ is an $m$-tuple of positive definite
matrices and $\mathbf w=(w_i)_{i=1}^m$ is a positive vector, then
\begin{equation}\label{eq:short-superadditive}
  \Sshort_V\left(\sum_{j=1}^m w_jA_j\right)
  \ge\sum_{j=1}^m w_j\Sshort_V(A_j).
\end{equation}
\end{lemma}

\begin{proof}
For positive definite matrices $E$ and $D$, the Schur-complement
criterion and order reversal under inversion give
\begin{equation}\label{eq:short-maximality-equivalence}
  E-VDV^*\ge0
  \quad\Longleftrightarrow\quad
  D^{-1}\ge V^*E^{-1}V
  \quad\Longleftrightarrow\quad
  D\le\Sshort_V(E).
\end{equation}
The first equivalence follows from the two Schur complements of the positive
block matrix
$\left[\begin{smallmatrix}E&VD\\DV^*&D\end{smallmatrix}\right]\ge0$.
Taking $E=A_j$ and $D=\Sshort_V(A_j)$ shows that
\begin{equation}
  A_j-V\Sshort_V(A_j)V^*\ge0.
\end{equation}
Now set
\begin{equation}
  A:=\sum_{j=1}^m w_jA_j,
  \qquad
  C:=\sum_{j=1}^m w_j\Sshort_V(A_j).
\end{equation}
Summing the preceding inequalities with weights $w_j$ gives
\begin{equation}\label{eq:superadditivity-residual}
  A-VCV^*
  =\sum_{j=1}^m w_j
    \bigl(A_j-V\Sshort_V(A_j)V^*\bigr)
  \ge0.
\end{equation}
Applying \eqref{eq:short-maximality-equivalence} to $E=A$ and $D=C$
now gives \eqref{eq:short-superadditive}.  This is the finite weighted
form of Anderson--Trapp's superadditivity theorem
\cite[Theorem~4]{AndersonTrapp}.
\end{proof}

The second property concerns the weighted geometric mean: shorting
before taking the mean can only decrease the result.  This is the
direction needed when the fixed-point equation is shorted later.

\begin{lemma}[Shorting and weighted geometric means]
\label{lem:short-geometric}
For $A,B>0$ and $0\le t\le1$,
\begin{equation}\label{eq:short-geometric}
  \Sshort_V(A\#_tB)
  \ge\Sshort_V(A)\#_t\Sshort_V(B).
\end{equation}
\end{lemma}

\begin{proof}
Self-duality, $(A\#_tB)^{-1}=A^{-1}\#_tB^{-1}$, order reversal under
inversion, and Fact~\ref{fact:ando-compression} give
\begin{align}
  \Sshort_V(A\#_tB)
  &=\bigl[V^*(A^{-1}\#_tB^{-1})V\bigr]^{-1}
    \label{eq:short-geometric-calculation}\\
  &\ge
  \bigl[(V^*A^{-1}V)\#_t(V^*B^{-1}V)\bigr]^{-1}\\
  &=\Sshort_V(A)\#_t\Sshort_V(B).
    \refstepcounter{equation}
    \tag*{\numberedqedstack{\theequation}}
\end{align}
\renewcommand{\qedsymbol}{}
\end{proof}

The third property controls the spectrum.  It says that a
$k$-dimensional shorted matrix cannot have smaller bottom eigenvalues
than the original matrix.

\begin{lemma}[Eigenvalues of a shorted matrix]
\label{lem:short-spectrum}
Suppose that $0<\beta_1\le\cdots\le\beta_n$ are the eigenvalues of
$B>0$ and $0<s_1\le\cdots\le s_k$ those of $\Sshort_V(B)$.  Then
\begin{equation}\label{eq:short-interlace}
  s_i\ge\beta_i,\qquad 1\le i\le k.
\end{equation}
\end{lemma}

\begin{proof}
Apply Fact~\ref{fact:cauchy-interlacing} to $V^*B^{-1}V$.  If
$\mu_1\ge\cdots\ge\mu_k$ are the eigenvalues of that
compression, then $\mu_i\le\lambda_i^\downarrow(B^{-1})=\beta_i^{-1}$
and $s_i=\mu_i^{-1}$, which proves \eqref{eq:short-interlace}.
\end{proof}

When the range of $V$ is itself a bottom spectral subspace, shorting
agrees with ordinary compression.  We use this simple special case
directly in the proof of the next theorem.

\section{Proof of the main theorem}
\label{sec:main-proof}

The results of the preceding section first give the positive-order
lower-product estimate.

\begin{theorem}[Lower weak log-majorization for positive power means]
\label{thm:positive-lower-log}
Let $\mathbf w=(w_i)_{i=1}^m$ be a positive probability vector, let
$\mathbf B=(B_j)_{j=1}^m$ be an $m$-tuple of positive definite matrices
in $\M_n$, and let $0<t\le1$.
Write $\beta_{j,1}\le\cdots\le\beta_{j,n}$ for the eigenvalues of
$B_j$ in increasing order, and put
$\boldsymbol\beta_i:=(\beta_{j,i})_{j=1}^m$.
Then, for every $1\le k\le n$,
\begin{equation}\label{eq:lower-products}
 \prod_{i=1}^k
 \lambda_i^{\uparrow}\bigl(P_t(\mathbf w;\mathbf B)\bigr)
 \ge
 \prod_{i=1}^k p_t(\mathbf w;\boldsymbol\beta_i).
\end{equation}
\end{theorem}

\begin{proof}
Fix $k$.  Let $V:\mathbb C^k\to\mathbb C^n$ be an isometry whose
range is spanned by eigenvectors of $P_t(\mathbf w;\mathbf B)$
corresponding to its $k$ smallest eigenvalues, and set
$X_k:=\Sshort_V\bigl(P_t(\mathbf w;\mathbf B)\bigr)$.
The range of $V$ reduces both $P_t(\mathbf w;\mathbf B)$ and its
inverse.  Thus shorting agrees with compression on this subspace, and
the eigenvalues of the compression are precisely the $k$ smallest
eigenvalues of the power mean.  Consequently,
\begin{equation}\label{eq:positive-mean-bottom-product}
  X_k=V^*P_t(\mathbf w;\mathbf B)V,
  \qquad
  \det X_k
  =\prod_{i=1}^k
   \lambda_i^\uparrow\bigl(P_t(\mathbf w;\mathbf B)\bigr).
\end{equation}

For each $j$, put $C_j:=\Sshort_V(B_j)>0$, and write
$\mathbf C:=(C_j)_{j=1}^m$.
Apply Lemma~\ref{lem:short-superadditive} and
Lemma~\ref{lem:short-geometric} to the fixed-point equation for
$P_t(\mathbf w;\mathbf B)$:
\begin{align}
 X_k
 &=\Sshort_V\left(\sum_jw_j
   \bigl(P_t(\mathbf w;\mathbf B)\#_tB_j\bigr)\right)
   \label{eq:shorted-supersolution}\\
 &\ge\sum_jw_j\Sshort_V
   \bigl(P_t(\mathbf w;\mathbf B)\#_tB_j\bigr)\\
 &\ge\sum_jw_j\left(
   \Sshort_V\bigl(P_t(\mathbf w;\mathbf B)\bigr)
   \#_t\Sshort_V(B_j)\right)\\
 &=\sum_jw_j(X_k\#_tC_j).
\end{align}
Thus $X_k$ is a supersolution for the $k$-dimensional power-mean
equation with inputs $\mathbf C$.

Let $s_{j,1}\le\cdots\le s_{j,k}$ be the eigenvalues of $C_j$.
Put $\mathbf s_i:=(s_{j,i})_{j=1}^m$.
Lemma~\ref{lem:short-spectrum} gives $s_{j,i}\ge\beta_{j,i}$ for
$1\le i\le k$.
Apply Proposition~\ref{prop:determinant} directly to this supersolution
in dimension $k$.  Reversing every eigenvalue list only permutes the
factors in its conclusion, so this inequality and scalar monotonicity yield
\begin{equation}\label{eq:localized-determinant-bound}
 \prod_{i=1}^k
 \lambda_i^{\uparrow}\bigl(P_t(\mathbf w;\mathbf B)\bigr)
 =\det X_k
 \ge\prod_{i=1}^k p_t(\mathbf w;\mathbf s_i)
 \ge\prod_{i=1}^k p_t(\mathbf w;\boldsymbol\beta_i).\qedhere
\end{equation}
\end{proof}

\begin{mainrestated}
Let $\mathbf w=(w_i)_{i=1}^m$ be a positive probability vector and let
$\mathbf A=(A_j)_{j=1}^m$ be an $m$-tuple of positive definite matrices
in $\M_n$.  For every $-1\le r<0$,
\begin{equation}\label{eq:main-restated}
 \left(\lambda_i^\downarrow\bigl(P_r(\mathbf w;\mathbf A)\bigr)\right)_{i=1}^n
 \prec_{w\log}
 \left(
   p_r\bigl(\mathbf w;\lambda_i^{\downarrow}(\mathbf A)\bigr)
 \right)_{i=1}^n.
\end{equation}
\end{mainrestated}

\begin{proof}
Write $a_{j,1}:=\lambda_1^{\downarrow}(A_j)\ge\cdots\ge
a_{j,n}:=\lambda_n^{\downarrow}(A_j)$, and put
$\mathbf a_i:=(a_{j,i})_{j=1}^m$.  The increasingly ordered eigenvalues
of $A_j^{-1}$ are $a_{j,i}^{-1}$.  Since $-r\in(0,1]$,
Theorem~\ref{thm:positive-lower-log}, applied to the inputs
$\mathbf A^{-1}$ and the positive order $-r$, gives
\begin{equation}\label{eq:positive-product-for-inverses}
  \prod_{i=1}^k
  \lambda_i^\uparrow\bigl(P_{-r}(\mathbf w;\mathbf A^{-1})\bigr)
  \ge
  \prod_{i=1}^k
  p_{-r}\bigl(\mathbf w;(a_{j,i}^{-1})_{j=1}^m\bigr)
  \qquad (1\le k\le n).
\end{equation}
The matrix and scalar power means both satisfy the same inverse duality:
\begin{equation}\label{eq:matrix-scalar-duality}
  P_r(\mathbf w;\mathbf A)
  =P_{-r}(\mathbf w;\mathbf A^{-1})^{-1},
  \qquad
  p_r(\mathbf w;\mathbf a_i)
  =p_{-r}\bigl(\mathbf w;(a_{j,i}^{-1})_{j=1}^m\bigr)^{-1}.
\end{equation}
The decreasingly ordered eigenvalues of the inverse in
\eqref{eq:matrix-scalar-duality} are the reciprocals of the increasingly
ordered eigenvalues before inversion.  Taking reciprocals in
\eqref{eq:positive-product-for-inverses} therefore gives
\begin{equation}\label{eq:duality-product-bound}
 \prod_{i=1}^k
 \lambda_i^{\downarrow}\bigl(P_r(\mathbf w;\mathbf A)\bigr)
 \le
 \prod_{i=1}^k p_r(\mathbf w;\mathbf a_i)
 \qquad (1\le k\le n).
\end{equation}
For every $j$, the sequence $i\mapsto a_{j,i}$ is decreasing, and the
scalar power mean is increasing in each argument.  Thus the terms on the
right of \eqref{eq:duality-product-bound} are already decreasing in
$i$.  Consequently, \eqref{eq:duality-product-bound} is precisely the
weak log-majorization \eqref{eq:main-restated}, and hence
\eqref{eq:main}.
\end{proof}

\begin{proof}[Proof of Corollary~\ref{cor:schatten-quasinorm}]
We use the weak-majorization form of Karamata's inequality: if
$x\prec_w y$ and $f$ is increasing and convex on an interval containing
their entries, then $\sum_i f(x_i)\le\sum_i f(y_i)$.  By
Theorem~\ref{thm:main}, the logarithms of the eigenvalues on the left of
\eqref{eq:main} are weakly majorized by the logarithms of the entries on
the right.  Applying this result with $f(u)=e^{qu}$ gives
\begin{equation}\label{eq:schatten-proof}
 \bigl\|P_r(\mathbf w;\mathbf A)\bigr\|_q^q
 \le \sum_{i=1}^n
 p_r\bigl(\mathbf w;(\lambda_i^\downarrow(A_j))_{j=1}^m\bigr)^q
 =\sum_{i=1}^n
 p_{r/q}\bigl(\mathbf w;((\lambda_i^\downarrow(A_j))^q)_{j=1}^m\bigr).
\end{equation}
Here the equality follows immediately from the definition of the scalar
power mean.

Since $r/q<0$, the scalar power mean $p_{r/q}$ is concave on the positive
cone; this follows by applying the Cauchy--Schwarz inequality to its
Hessian.  Its positive homogeneity then implies superadditivity.  Applying
this to the summands in \eqref{eq:schatten-proof} yields
\begin{align}
 \sum_{i=1}^n
 p_{r/q}\bigl(\mathbf w;((\lambda_i^\downarrow(A_j))^q)_{j=1}^m\bigr)
 &\le
 p_{r/q}\left(\mathbf w;
   \left(\sum_{i=1}^n(\lambda_i^\downarrow(A_j))^q\right)_{j=1}^m\right)
   \label{eq:schatten-superadditive}\\
 &=p_r\bigl(\mathbf w;(\|A_j\|_q)_{j=1}^m\bigr)^q.
   \label{eq:schatten-final}
\end{align}
Combining these inequalities and taking the $q$-th root proves
\eqref{eq:schatten-quasinorm}.
\end{proof}

\newpage
\appendix
\section{Lean correspondence}
\label{app:lean-correspondence}

The accompanying Lean development~\cite{TomamichelFormalization} follows the manuscript proof.
Figure~\ref{fig:lean-dependencies} gives the dependency graph (arrows point
to ingredients and blue boxes mark the four published inputs), while
Table~\ref{tab:lean-correspondence} gives the statement-by-statement
correspondence.

Every statement in Table~\ref{tab:lean-correspondence} is formalized in
Lean using Mathlib and checked by the Lean kernel.  The four published
inputs (Definition~\ref{def:power-means} and
Facts~\ref{fact:ando-compression}, \ref{fact:multiplicative-lidskii}, and
\ref{fact:cauchy-interlacing}) are proved in the accompanying development,
so none is assumed as an axiom or hypothesis in the final declarations
(Theorem~\ref{thm:main} and
Corollary~\ref{cor:schatten-quasinorm}).  No project-specific axioms are
introduced.
The measure extension in
Remark~\ref{rem:probability-measures} is outside
the development.

\medskip

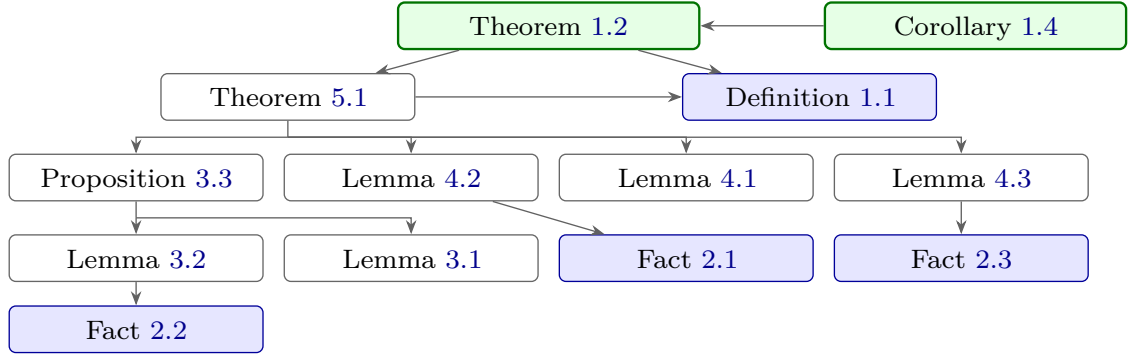
\begin{figure}[!ht]
\centering
\makebox[\textwidth][c]{%
\resizebox{0.95\linewidth}{!}{%
\begin{tikzpicture}[
  proofnode/.style={draw=black!60, rounded corners=2pt, align=center,
    fill=white, inner sep=2.5pt, text width=2.5cm, text height=1.55ex,
    text depth=0.35ex, font=\scriptsize},
  mainnode/.style={proofnode, fill=green!10, draw=green!45!black,
    line width=0.7pt, text width=3.0cm, font=\scriptsize},
  boundary/.style={proofnode, fill=blue!10, draw=blue!60!black},
  edge/.style={-{Stealth[length=1.5mm]}, thin, draw=black!62},
  every path/.style={edge}
]
  \node[mainnode] (main) at (0,0) {Theorem~\ref{thm:main}};
  \node[mainnode] (schatten) at (4.5,0)
    {Corollary~\ref{cor:schatten-quasinorm}};

  \node[proofnode] (positive) at (-2.75,-0.75)
    {Theorem~\ref{thm:positive-lower-log}};
  \node[boundary] (powerdef) at (2.75,-0.75)
    {Definition~\ref{def:power-means}};

  \node[proofnode] (detsuper) at (-4.35,-1.6)
    {Proposition~\ref{prop:determinant}};
  \node[proofnode] (geoshort) at (-1.45,-1.6)
    {Lemma~\ref{lem:short-geometric}};
  \node[proofnode] (shortsum) at (1.45,-1.6)
    {Lemma~\ref{lem:short-superadditive}};
  \node[proofnode] (shortspec) at (4.35,-1.6)
    {Lemma~\ref{lem:short-spectrum}};

  \node[proofnode] (lidtrace) at (-4.35,-2.45)
    {Lemma~\ref{lem:lidskii-trace}};
  \node[proofnode] (supernorm) at (-1.45,-2.45)
    {Lemma~\ref{lem:normalized-supersolution}};
  \node[boundary] (ando) at (1.45,-2.45)
    {Fact~\ref{fact:ando-compression}};
  \node[boundary] (interlace) at (4.35,-2.45)
    {Fact~\ref{fact:cauchy-interlacing}};

  \node[boundary] (lidskii) at (-4.35,-3.2)
    {Fact~\ref{fact:multiplicative-lidskii}};

  \draw (main) -- (positive);
  \draw (main) -- (powerdef);
  \draw (schatten) -- (main);
  \draw (positive) -- (powerdef);

  \coordinate (pbus) at (-2.75,-1.175);
  \draw[-] (positive.south) -- (pbus);
  \draw (pbus) -| (detsuper.north);
  \draw (pbus) -| (geoshort.north);
  \draw (pbus) -| (shortsum.north);
  \draw (pbus) -| (shortspec.north);

  \coordinate (dbus) at (-4.35,-2.025);
  \draw[-] (detsuper.south) -- (dbus);
  \draw (dbus) -| (lidtrace.north);
  \draw (dbus) -| (supernorm.north);
  \draw (lidtrace) -- (lidskii);

  \draw (geoshort) -- (ando);
  \draw (shortspec) -- (interlace);
\end{tikzpicture}
}%
}
\caption{Dependency graph for the Lean formalization.}
\label{fig:lean-dependencies}
\end{figure}

\setcounter{table}{1}
\begin{table}[!ht]
\centering
\begingroup
\setlength{\tabcolsep}{3pt}
\renewcommand{\arraystretch}{1.02}
\begin{tabular}{@{}
  >{\raggedright\arraybackslash}p{0.2\textwidth}
  >{\raggedright\arraybackslash}p{0.55\textwidth}
  >{\raggedright\arraybackslash}p{0.15\textwidth}@{}}
\toprule
\textbf{Manuscript statement} & \textbf{Lean declaration} & \textbf{Status}\\
\midrule
Definition~\ref{def:power-means}
& \nolinkurl{FixedPoint}; \nolinkurl{powerMean};
  \nolinkurl{negativePowerMean}
& Formalized \\

Theorem~\ref{thm:main}
& \nolinkurl{negativePowerMean_weakLogMajorized}
& Formalized\\

Corollary~\ref{cor:schatten-quasinorm}
& \nolinkurl{negativePowerMean_schattenQ_le}
& Formalized\\

Fact~\ref{fact:ando-compression}
& \nolinkurl{Ando}
& Formalized \\

Fact~\ref{fact:multiplicative-lidskii}
& \nolinkurl{Lidskii}
& Formalized \\

Fact~\ref{fact:cauchy-interlacing}
& \nolinkurl{Interlacing}
& Formalized \\

Lemma~\ref{lem:normalized-supersolution}
& \nolinkurl{IsSuper.norm}
& Formalized\\

Lemma~\ref{lem:lidskii-trace}
& \nolinkurl{Lidskii.trace}
& Formalized\\

Proposition~\ref{prop:determinant}
& \nolinkurl{IsSuper.powerCoordinateProd_le_det}
& Formalized\\

Lemma~\ref{lem:short-superadditive}
& \nolinkurl{shorted_superadditive}
& Formalized\\

Lemma~\ref{lem:short-geometric}
& \nolinkurl{weightedGeometricMean_shorted_le}
& Formalized\\

Lemma~\ref{lem:short-spectrum}
& \nolinkurl{increasingEigenvalues_le_shorted}
& Formalized\\

Theorem~\ref{thm:positive-lower-log}
& \nolinkurl{powerMean_lowerProductDominates}
& Formalized\\
\bottomrule
\end{tabular}
\endgroup
\caption{Correspondence between manuscript statements and Lean declarations.}
\label{tab:lean-correspondence}
\end{table}

\end{document}